\documentclass[a4paper,11pt]{article}
\usepackage[plainpages=false]{hyperref}
\usepackage{amsfonts,latexsym,rawfonts,amsmath,amssymb,amsthm,mathrsfs}
\usepackage{amsmath,amssymb,amsfonts,latexsym,lscape,rawfonts}
\usepackage[usenames]{color}

\usepackage[all]{xy}
\usepackage{eufrak}
\usepackage{makeidx}         
\usepackage{graphicx,psfrag}

\usepackage{array,tabularx}

\usepackage{setspace}

\newtheorem{claim}{Claim}[section]
\newtheorem{thm}{Theorem}[section]
\newtheorem{cor}[thm]{Corollary}
\newtheorem{lemma}[thm]{Lemma}

\newtheorem{prop}[thm]{Proposition}

\theoremstyle{remark}
\newtheorem{rmk}[thm]{Remark}

\theoremstyle{definition}
\newtheorem{Def}[thm]{Definition}

\title {On the Extension of the Calabi Flow on Toric Varieties}
\date{}
\author{Hongnian Huang}
\begin{document}

\maketitle

\begin{abstract}
Inspired by recent work of S. K. Donaldson on constant scalar curvature metrics
on toric complex surfaces, we study obstructions to the extension of the Calabi flow on a polarized
toric variety.  Under some technical assumptions, we prove that the Calabi flow
can be extended for all time.
\end{abstract}

\section{Introduction}

In \cite{Ca1}, E. Calabi proposed to deform a given K\"ahler metric in the direction of the Levi-Hessian of its scalar curvature. This is a 4th order fully nonlinear semiparabolic
equation aiming to attack the existence of constant scalar curvature K\"ahler metric (cscK for short) in a given K\"ahler class. Note that cscK metrics are the fixed points while extremal K\"ahler metrics are soliton solutions of the Calabi flow. The Calabi conjecture on the existence of K\"ahler-Einstein metrics, as well as Yau-Tian-Donaldson's
conjecture on the existence of extremal K\"ahler metrics are central problems in K\"ahler geometry. However, by fixing a maximal torus in the complex automorphism group, an extremal K\"ahler metric satisfies a 4th
order nonlinear partial differential equation. It is hard to attack the existence problem directly. The study of the Calabi flow seems to be an effective approach,
although rather a complicated one. In \cite{Chen1}, Xiuxiong Chen conjectured that the flow exists globally (i.e. for all time) for any smooth initial K\"ahler metric.

Unfortunately, at the moment very little is known about the global existence of the Calabi flow. In the Riemannian surface case, it was settled down by the work of P. Chrusci\'el \cite{Chr} (c.f. \cite{Chen1} also). In a subsequent paper \cite{ChenHe}, Xiuxiong Chen and Weiyong He proved that the main obstruction to the global existence of the Calabi flow is the bound of the Ricci
curvature. Other important results about Calabi flow appear in Chen-He \cite{ChenHe2}, \cite{ChenHe3}; W.Y. He \cite{He}; J. Fine \cite{JF}; G. Sz\^{e}kelyhidi \cite{GA}.

S.K. Donaldson \cite{D1} set up a program to prove the existence of an extremal K\"ahler metric on a toric surface under the $K$-stability assumption and he completed this program in the cscK case in \cite{D4}. Let us briefly introduce  Donaldson's results.

\begin{thm}(\cite{D4})
\label{M1}
Any polarized complex toric surface with zero Futaki invariant is $K$-stable if and only if it admits a constant scalar curvature K\"ahler metric.
\end{thm}

His strategy of proof is as follows: if we consider the Legendre transform $u$ of a K\"ahler potential $\phi$, then we obtain a convex function $u$ on the polytope $P \subset \mathbb{R}^2$ associated to the toric variety, satisfying certain boundary conditions defined by a measure $\sigma$ on each facet of $P$ (see Definition \ref{G}). The Abreu's equation tells us that the metric is extremal if and only if $${u^{ij}}_{ij}=-A,$$ where $A$ is an affine function determined by the data $(P, \sigma)$. We need to solve this $4$-th order differential equation when $A$ is a constant and $u$ is a smooth convex function and verifies the given boundary conditions. Then, Donaldson applied the continuity method to prove the following theorem (c.f. Theorem 1 in \cite{D3}) :

\begin{thm} (\cite{D3})
\label{M2}
Let $(P(\alpha), \sigma(\alpha),A(\alpha))$ be a sequence of polytopes converging to $(P,\sigma,A)$ where  the number of edges of $P(\alpha)$ does not depend on $\alpha$. Suppose that for each $\alpha$ there is a solution $u(\alpha)$ to the problem defined by $(P(\alpha), \sigma(\alpha),A(\alpha))$, i.e.
$$
u(\alpha)^{ij}_{\ ij} = -A(\alpha),
$$ where $A(\alpha)$ is a constant and $u(\alpha)$ satisfies the Guillemin boundary conditions of $(P(\alpha), \sigma(\alpha))$.
If there is an $M > 0$ such that each $u(\alpha)$ satisfies the $M$-condition given by Definition \ref{M} below,  then there is a solution of the problem defined by $(P,A,\sigma),$ i.e. there is a smooth function $u$ such that
$$
u^{ij}_{\ ij} = -A
$$
and $u$ satisfies the Guillemin boundary conditions of $(P, \sigma)$.
\end{thm}

Inspired by this work of Donaldson, we attempt to study the Calabi flow on an $n$-dimensional toric variety. The main theorem we obtain is

\begin{thm}
\label{Calabi flow} For any toric K\"ahler variety, the Calabi flow (initiated from any toric invariant
metric) can be extended indefinitely as long as the following assumptions hold

\begin{itemize}
\item The L$^n$-norm of the Riemannian curvature is bounded for any finite time interval $[0,T)$.
\item At each time $t \in [0, T)$, after rescaling the metric by $|Rm|_{\infty}$, the first derivative of the Riemannian curvature is uniformly bounded.
\item The Euclidean norm of the gradient of scalar curvature on the polytope i.e. $|\nabla R|$, are uniformly bounded for all time.
\end{itemize}

\begin{rmk}
For K\"ahler surfaces, the first assumption is automatically true since the Calabi flow decreases the Calabi energy, hence the $L^2$-norm of Riemannian curvature. To get rid of the second assumption, one needs to extend Shi's Ricci flow pointwise curvature estimates \cite{Shi} to the Calabi flow. The last condition is imposed to guarantee that the $M$-condition holds along the Calabi flow. If one can extend Perelman's Ricci flow non-collapsing result \cite{Pe} to the Calabi flow, then the third assumption is not needed. 
\end{rmk}

\end{thm}
The organization of this paper is as follows: In section 2, we set up the notations for the K\"ahler geometry and the Calabi flow as in \cite{ChenHe}. In section 3, we give a brief introduction to polarized toric varieties and study the Calabi flow in the corresponding polytope. In section 4, we extend Donaldson's geometrical estimates \cite{D3} from dimension two to higher dimensions. In section 5, we rule out the singularities in the Calabi flow under the assumptions of Theorem \ref{Calabi flow}.  \\

{\bf Acknowledgment} This paper is a part of the author's PhD thesis. The author would like to thank his advisor Xiuxiong Chen for drawing his attention to this problem and sharing insightful ideas with him. He also wants to thank Professor Pengfei Guan and Professor Vestislav Apostolov for many useful suggestions. He is grateful to Fang Yuan and Song Sun for discussions on related topics. He would like to express his gratitude to the referee for careful readings and many suggestions.

\section{K\"ahler geometry and Calabi flow}

Let $M$ be a compact complex manifold of complex dimension $n$. A Hermitian metric metric $g$ on $M$ in local coordinates is given by
$$
g=g_{i \bar{j}} d z^i \otimes d z^{\bar{j}}
$$
where $\{ g_{i \bar{j}} \}$ is a positive definite Hermitian matrix with smooth dependence on the coordinates. We use $\{ g^{i \bar{j}} \}$ to denote the inverse matrix of $\{ g_{i \bar{j}} \}$. The K\"ahler condition says that the corresponding K\"ahler form $\omega = \sqrt{-1} g_{i\bar{j}} dz^i \wedge dz^{\bar{j}}$ is a closed $(1,1)$ form. The K\"ahler class of $\omega$ is its cohomology class $[\omega] \in H^2(M,\mathbb{R})$. By  Hodge theory, any other K\"ahler form in the same class is of the form
$$
\omega_{\phi} = \omega + \sqrt{-1} \partial \bar{\partial} \phi > 0,
$$
for some real valued function $\phi$ on $M$, where
$$
\partial \bar{\partial} \phi = \sum_{i,j=1}^n \partial_i \partial_{\bar{j}} \phi d z^i \wedge d z^{\bar{j}} = \phi_{,i\bar{j}} dz^i \wedge d z^{\bar{j}}.
$$

The corresponding K\"ahler metric is denoted by $g_{\phi} = (g_{i \bar{j}} + \phi_{,i \bar{j}}) d z^i \otimes d z^{\bar{j}}$, and we use $\{g^{i \bar{j}}_{\phi} \}$ to denote the inverse matrix of $\{g_{i \bar{j}} + \phi_{,i \bar{j}} \}$. For simplicity, we use both $g$ and $\omega$ to denote the K\"ahler metric. The space of K\"ahler potentials is defined to be
$$
\mathcal{H}_{\omega} = \{\phi \in C^{\infty} (M) | \omega_{\phi} = \omega + \sqrt{-1} \partial \bar{\partial}\phi > 0 \},
$$
which is identified with the space of K\"ahler metrics and is the main objects we are interested in.

Given a K\"ahler metric $\omega$, its volume form is
$$
\omega^n = \frac{(\sqrt{-1})^n}{n!} \det(g_{i\bar{j}}) dz^1 \wedge d z^{\bar{1}} \wedge \cdots \wedge d z^n \wedge d z^{\bar{n}}.
$$
The Ricci curvature of $\omega$ is locally given by
$$
R_{i\bar{j}} = - \partial_i \partial_{\bar{j}} \log \det (g_{k\bar{l}}),
$$
the Ricci form being
$$
Ric_{\omega} = \sqrt{-1} R_{i\bar{j}} d z^i d z^{\bar{j}} = - \sqrt{-1} \partial_i \partial_{\bar{j}} \log \det (g_{k\bar{l}}).
$$
It is a real, closed (1,1) form. The cohomology class of the Ricci form is the first Chern class $C_1(M)$, and is therefore independent of the metric.

Given a polarized compact K\"ahler manifold $(M, [\omega])$, for any $\phi \in \mathcal{H}$, Calabi \cite{Ca1}, \cite{Ca2} introduced the Calabi functional,
$$
Ca(\omega_{\phi}) = \int_M R^2_{\phi} \omega^n_{\phi},
$$
where $R_{\phi}$ is the scalar curvature of $\omega_{\phi}$. Note that both the total volume
$$
V_{\phi} = \int_M \omega_{\phi}^n
$$
and the total scalar curvature
$$
S_{\phi} = \int_M R_{\phi} \omega_{\phi}^n
$$
remain unchanged when $\phi$ varies in $\mathcal{H}_{\omega}$. As a consequence, the average scalar curvature
$$
\underline{R} = \frac{S_{\phi}}{V_{\phi}}
$$
is a constant depending only on the class $[\omega]$. Usually we use the following modified Calabi energy
$$
\widetilde{Ca} (\omega_{\phi}) = \int_M (R_{\phi} - \underline{R})^2 \omega_{\phi}^n
$$
to replace $Ca(\omega_{\phi})$ since they only differ by a topological constant. E. Calabi studied the variational problem to minimize $Ca(\omega_{\phi})$ in $\mathcal{H}_{\omega}$. The critical points turn out \cite{Ca2} to be either CscK or extremal K\"ahler metric depending on whether the Futaki character vanishes or not. The Futaki character $f = f_{\phi} : \mathfrak{h} (M) \rightarrow \mathbb{C}$ is defined on the Lie algebra $\mathfrak{h} (M)$ of all holomorphic vector fields of $M$ as follows,
$$
f_{\phi} (X) = - \int_M X(F_{\phi}) w_{\phi}^n,
$$
where $ X \in \mathfrak{h}$ and $F_{\phi}$ is a real valued function defined by
$$
F_{\phi} = G_{\phi} ( R_{\phi}).
$$
$G_{\phi}$ is the Hodge-Green integral operator, and $F_{\phi}= G_{\phi} ( R_{\phi}) $ is equivalent to $\triangle_{\phi} F_{\phi} = R_{\phi} - \underline{R},$ where $\triangle_{\phi}$ is the Laplace operator of the metric $\omega_{\phi}$. In \cite{Ca2}, E. Calabi showed that the Futaki character $f=f_{\phi}$ is invariant when $\phi$ varies in $\mathcal{H}_{\omega}$.

The existence of CscK metrics (or extremal K\"ahler metrics) seems intractable at the first glance since the equation is a fully nonlinear 4th order equation. In \cite{Ca1}, E. Calabi proposed the so-called Calabi flow to approach the existence problem. The Calabi flow is the gradient flow of the Calabi functional, defined as the following parabolic equation with respect to a real parameter $t \geq 0$,

$$
\frac{\partial}{\partial t} g_{i\bar{j}}(t) = \partial_i \partial_{\bar{j}} R_{g(t)}.
$$
On the potential level, the Calabi flow is of the form
$$
\frac{\partial \phi}{\partial t} = R_{\phi} - \underline{R}.
$$
Under the Calabi flow, we have
$$
\frac{d}{dt} \int_M (R_{\phi} - \underline{R})^2 \omega_{\phi}^n = -2 \int_M (D_{\phi} R_{\phi}, R_{\phi}) \omega_{\phi}^n,
$$
where $D_{\phi}$ is the Lichn\'{e}rowicz operator defined by
$$
D_{\phi} f = f_{,\alpha \beta}^{\quad \alpha \beta},
$$
and where the covariant derivative is with respect to $\omega_{\phi}$. So the Calabi energy is strictly decreasing along the flow unless $\omega_{\phi}$ is an extremal K\"ahler or a CscK metric.

\section{Toric geometry and Calabi flow}
In this section, we use the description of toric manifolds which is due to Guillemin \cite{G1} \cite{G2} and Abreu \cite{A2}, see also Donaldson \cite{D1} and Apostolov-Calderbank-Gauduchon \cite{ACG}. Given a $n$-dimensional polarized toric variety $X$ with K\"ahler form $\omega$ and Hamiltonian action of an $n$-dimensional torus action $\mathbb{T}^n$, we denote its moment map by $\mu$. The image of the moment map is a Delzant polytope $P$ in $\mathbb{R}^n$. Let $X_0 = \mu^{-1} (P_0)$, where $P_0$ is the interior of $P$; $X_0$ is a dense open subset of $X$ diffeomorphic to $\mathbb{R}^n \times \mathbb{T}^n$. Also, the preimage of each boundary face of $P$ corresponds to a divisor of $X$. A model case is $\mathbb{CP}^2$; up to an appropriate normalization, the image of its moment map can be taken to be the triangle in $\mathbb{R}^2$ with vertices $(0,0),(0,1),(1,0)$ and the preimage of each facet is a projective line $\mathbb{CP}^1 \subset \mathbb{CP}^2$.

The K\"ahler form $\omega$ restricted to $X_0$ can be written as 

$$
\omega= \sqrt{-1} \phi_{,ij} \ d z^i \wedge d z^{\bar{j}},
$$

where $z_i=\xi_i + \sqrt{-1} \eta_i,\ \xi_i \in \mathbb{R}^n,\ \eta_i \in \mathbb{T}^n$ and 

$$
\phi_{,ij}=\frac{\partial^2 \phi}{\partial \xi_i \partial \xi_j}.
$$

We can write down the moment map explicitly in this case: $$x=\mu (z)=\mu (\xi)=(\frac{\partial \phi}{\partial \xi_1},\ldots,\frac{\partial \phi}{\partial \xi_n} ).$$Using the Legendre transformation, we obtain the symplectic potential $u$ on $P_0$: For each point $x \in P_0$, there is a unique point $\xi \in \mathbb{R}^n$ such that $\frac{\partial \phi}{\partial \xi_i} = x_i$ and we let 

$$
u(x) = \sum x_i \xi_i - \phi(\xi).
$$

It is important to point out that $u(x)$ should satisfy the Guillemin boundary conditions by Abreu \cite{A2}, Donaldson \cite{D1} and Apostolov-Calderbank-Gauduchon \cite{ACG}. Let $d$ be the number of $(n-1)$-dimensional faces of $P$, we can describe the polytope $P$ by a set of inequalities

$$
\l_i(x)=\langle x, u_i \rangle - \lambda_i \geq 0, \ i=1, \ldots, d;
$$
the $u_i$ being primitive elements of the lattice $\mathbb{Z}^n$.

Let 
$$
u_0(x) = \frac{1}{2} \sum^d_{k=1} l_k(x) \log l_k(x).
$$

\begin{Def} {\label G}
$u(x)$ satisfies the Guillemin boundary conditions if and only if $u(x)-u_0(x)$ is a smooth function on $P_0$ up to the boundary and $u(x)$ restricted to each facet is smooth and strictly convex.
\end{Def}

Next we want to write down the formula for Riemannian curvature in symplectic coordinates which is due to Donaldson \cite{D1} and Abreu \cite{A1}. In symplectic coordinates the metric is given by $$g = \sum u_{,ij} d x_i d x_j + \sum u^{,ij} d \xi_i d \xi_j, $$ where the matrix
$(u^{,ij})$ is the inverse of the Hessian matrix $u_{,ij} = \frac{\partial^2 u}{\partial x_i \partial x_j}$. 

Now we define a 4-index tensor by $$F^{ab}_{kl}=-{u^{ab}}_{kl}.$$ We can raise and lower indices in the usual way, using
the metric $u_{ij}$, setting $$F^{abcd}=u^{ck}u^{dk}F^{ab}_{kl}, \quad F_{ijkl}=u_{ia}u_{jb}F^{ab}_{kl}.$$

\begin{lemma} (Donaldson \cite{D1})
The curvature tensor of $g$ is $$-F^{ijkl}d z_i d \bar{z}_k \otimes d z_j d\bar{z}_l.$$
\end{lemma}

Using the metric $u_{ij}$, the standard square-norm of the tensor $F$ is
$$|F|^2=F^{ij}_{kl}F^{ab}_{cd}u_{ia}u_{jb}u^{kc}u^{ld}={u^{ab}}_{cd}{u^{cd}}_{ab}.$$

Following \cite{A1}, the scalar curvature can be written as: $$S(u)=-\sum_{ij}\frac{\partial^2 u^{,ij}}{\partial x_i \partial x_j}. $$

Let us see what the evolution equation of $u$ under the Calabi flow would be. Since the evolution equation for $\phi$ is $$\frac{\partial \phi}{\partial t}=R - \underline{R} = A - \underline{A},$$ where $\underline{R}=\underline{A}$ is the average of scalar curvature $R$, 
 Donaldson shows that \cite{D1}, $$\frac{\partial u}{\partial t}=-\frac{\partial \phi}{\partial t}=-A+\underline{A}.$$
 
In \cite{D3}, Donaldson introduces the $M$-condition to control the injectivity radius of $X$:
 
\begin{Def} {\label M}
Let $p, q$ be distinct points in the interior of $P$. Let $\nu$ be the unit vector pointing in the
direction from $p$ to $q$. We write $$V(p,q)=(\nabla_{\nu} u)(q)-(\nabla_{\nu} u)(p),$$ where $\nabla_{\nu}$ denotes
the derivative in the direction $\nu$. Thus $V(p,q)$ is positive by the convexity condition. Let $I(p,q)$ be the line
segment $$I(p,q)=\{\frac{p+q}{2}+t(p-q):-3/2 \leq t \leq 3/2\}.$$

For $M > 0$ we say that the symplectic potential $u$ satisfies the $M$-condition if for any $p,q$ such that $I(p,q) \subset P$ we have $V(p,q) \leq M$.
\end{Def}

\begin{lemma}
\label{Init}
For any convex function $u$ on a polytope $P \subset \mathbb{R}^n$ satisfying the Guillemin
boundary conditions, there exits a constant $M$ such that $u$ satisfies the $M$-condition.
\end{lemma}
\begin{proof}
Near any boundary of $P$, $u$ can be expressed as $$u = \frac{1}{2} (l_{E_1} \log l_{E_1} + \cdots + l_{E_m} \log l_{E_m}) + f$$ where $E_i$ is a facet of $P$ and $l_{E_i}$ is the defining function for $E_i$. Without loss of generality, we only need to prove that for the function
\begin{eqnarray*}
u &:& \mathbb{R}^m \to \mathbb{R}\\
u(x_1,\ldots,x_m)&=&x_1 \log x_1 + \cdots + x_m \log x_m
\end{eqnarray*}
and for any point $(x_1,\ldots,x_m)$ where $x_i >0$ for all $i$, the difference of derivatives of $u$ at $(x_1,\ldots,x_m)$ and $(2x_1,\ldots,2x_m)$ in the direction of the unit vector 
$$
v = \frac{(-x_1, \dots, -x_m)}{\sqrt{x_1^2 + \cdots + x_m^2}}
$$ is uniformly bounded. In fact
\begin{eqnarray*}
& & \ \nabla_{\nu}u(x_1,\ldots,x_m)-\nabla_{\nu}u(2x_1,\ldots,2x_m)\\
&=& (\log x_1 - \log 2x_1)\frac{-x_1}{\sqrt{x_1^2+\cdots+x_m^2}} + \\ 
& & \cdots + (\log x_m - \log 2x_m)\frac{-x_m}{\sqrt{x_1^2+\cdots+x_m^2}}\\
&\leq& m\log 2.
\end{eqnarray*}
\end{proof}

\section{Geometric Estimates}
The estimates in this section follow Donaldson's work in toric surfaces \cite{D3}. However, there are some lemmas needed to be rewritten in order to deal with the cases when $n > 2$. For the reader's convenience, we put together our results with Donaldson's work.

\begin{lemma}
\label{Riemannian distance controlled by Euclidean distance} Suppose $u$ satisfies the $M$-condition. Let $I$ be a line
segment in $\bar{P}$ with mid-point $p$ and let $p'$ be an end point of $I$. Then the Riemannian length of the segment
$pp'$ is at most $$\frac{1}{\sqrt{2}-1}\sqrt{M |p-p'|_{\rm Euc}}.$$
\end{lemma}

\begin{proof}
We can suppose that $p'$ is the origin and that $p$ is $(L,0)$, so $|p-p|_{{\rm Euc}} = L$ and the segment of the
$x_1$-axis from $0$ to $2L$ lies in $\bar{P}$. We apply the definition of the $M$-condition to the pair of points
$p,q$, where $q=(L/2,0).$ This gives $$\int^{L}_{L/2} u_{11}(t,0) dt \leq M.$$

The Riemannian length of the straight line segment from $q$ to $p$ is $$\int^{L}_{L/2} \sqrt{u_{11}}(t,0) dt$$ which is at
most $$\sqrt{L/2} \left ( \int^{L}_{L/2} u_{11}(t,0) dt \right )^{\frac{1}{2}}.$$ Hence the Riemannian length of this
segment is at most $\sqrt{LM/2}$. Replacing $p$ by $2^{-r}p$ and summing over $r$ we see that the Riemannian length of
the segment from 0 to $p$ is at most $$\sqrt{ML}\sum_{r=1}^{\infty}(\frac{1}{\sqrt{2}})^r,$$ from which the result
follows.
\end{proof}

\begin{cor}
Suppose that $u$ satisfies the $M$-condition and that $p$ is a point of $P$. Then $${\rm Dist}_g(p,\partial P) \leq
\frac{1}{\sqrt{2}-1}\sqrt{M {\rm Dist}_{\rm Euc}(p,\partial P)}.$$
\end{cor}

\begin{proof}
To see this we take $p'$ to be the point on $\partial P$ closest to $p$, in the Euclidean metric. If $p''=2p-p'$ then
the segment $p'p''$ lies in $\bar{P}$ and we can apply the lemma above.
\end{proof}

Next we derive a crucial result which relates the restriction of $u$ to lines and the curvature tensor $F$.

\begin{lemma}
At each point $p$ of $P$,
\begin{equation}
\label{Curvature inequality}
\left ( \frac{\partial}{\partial x_1} \right )^2 (u^{-1}_{11})(p) \leq |F|(p).
\end{equation}
\end{lemma}

\begin{proof}
Observe that the quantity $$\left ( \frac{\partial}{\partial x_1} \right )^2 (u^{-1}_{11})(p)$$ is unchanged by rescaling
$x_1$. So by scaling $x_1$, we will get $u_{11}(p)=1$. Next we want to show that after carefully selecting $x_2,\ldots,x_n$, $u_{ij}(p)$ will be a standard Euclidean metric. We need a standard linear algebra fact.

\begin{claim}
Suppose that $(u_{ij})$ is a symmetric and positive definite matrix. Then there is an upper triangular matrix\[
A = \left(
\begin{array}{ccccc}
1 & a_{12} & a_{13} & \cdots & a_{1n} \\
0 & 1 & a_{23} & \cdots & a_{2n} \\
& & \ddots & &\\
0 & 0 & \cdots & 1 & a_{n-1 \ n}\\
0 & 0 & \cdots & 0 & 1
\end{array}
\right)
\]
such that
\[
A^{T} (u_{ij}) A = B = \left(
\begin{array}{cccc}
\lambda_1 & 0 & \cdots & 0\\
0 & \lambda_2 & \cdots & 0\\
& & \ddots & \\
0 & \cdots & 0 & \lambda_n
\end{array}
\right), 
\]
where $B$ is a diagonal matrix with $\lambda_1, \cdots, \lambda_n > 0$, more importantly $\lambda_1=u_{11}$.
\end{claim}

Set $p = (p_1,\ldots,p_n) \in P$, $A$ be the matrix in the previous lemma and $v(x)=u(p+(x-p)A^T)$. Then
the $i$-th element of $(x-p)A^T$ is $$(x_1-p_1)a_{i1}+(x_2-p_2)a_{i2}+\cdots+(x_n-p_n)a_{in}.$$
So
\begin{eqnarray*}
\frac{\partial^2 v}{\partial x_i \partial x_j} &=& \frac{\partial}{\partial x_j}\left(\frac{\partial v}{\partial x_i}\right)\\
& = & \frac{\partial}{\partial x_j} (a_{1i}u_1(p+(x-p)A) + a_{2i}u_2(p+(x-p)A) + \\
&  & \cdots + a_{ni}u_n(p+(x-p)A)) \\
& = & (a_{1i}, a_{2i}, \cdots, a_{ni}) \left(
\begin{array}{c}
\frac{\partial}{\partial x_j}(u_1(p+(x-p)A)) \\
\frac{\partial}{\partial x_j}(u_2(p+(x-p)A)) \\
\vdots\\
\frac{\partial}{\partial x_j}(u_n(p+(x-p)A)) \\
\end{array}
\right)\\
& = & (a_{1i}, a_{2i}, \cdots, a_{ni}) \left(
\begin{array}{ccc}
u_{11} & \cdots & u_{1n} \\
& \vdots\\
u_{n1} & \cdots & u_{nn}\\
\end{array}
\right)
\left(
\begin{array}{c}
a_{1j}\\
\vdots\\
a_{nj}\\
\end{array}
\right),
\end{eqnarray*}
which means
$$
\left( \frac{\partial^2 v}{\partial x_i \partial x_j} \right)(p) = A^T (u_{ij}) A.
$$
Hence $(\frac{\partial^2 v}{\partial x_i \partial x_j})(p)$ is a diagonal matrix with $v_{11}(p)=u_{11}(p)=1$,
more importantly, $v$ restricts on the line $\{ p + t(1,0,\ldots,0) | t \in \mathbb{R} \}$ is the same as $u$ since
$$
v(p+(t,0,\ldots,0))=u(p+(t,0,\ldots,0)A^T)=u(p+(t,0,\ldots,0)).
$$
It is easy to see after rescaling $x_2,\ldots,x_n$, that $(v_{ij})$ can be the identity matrix.

At this point, after changing the coordinate system the left hand side of inequality (\ref{Curvature inequality}) remains unchanged. We also want to show that the right hand side doesn't change either. Let us use indices $\alpha, \beta$ representing the new coordinate system and $i, j$ representing the old coordinate system. Also we set $A^{-1}=(b_{ij})$. From the above calculation, we get $$
v^{\alpha \beta} (p) = b_{\alpha i} \ u^{ij}(p) \ b_{\beta j}.
$$

Hence at point $p$, we have
$$
v^{\alpha \beta}_{\quad \gamma \delta}=a_{k \gamma} b_{\alpha i} u^{ij}_{\quad kl} b_{\beta j} a_{l \delta}
$$

and
$$
v^{\gamma \delta}_{\quad \alpha \beta}=a_{\bar{i} \alpha} b_{\gamma \bar{k}} u^{\bar{k} \bar{l}}_{\quad \bar{i} \bar{j}} b_{\delta \bar{l}} a_{\bar{j} \beta}.
$$

Then

\begin{eqnarray*}
& & v^{\alpha \beta}_{\quad \gamma \delta} v^{\gamma \delta}_{\quad \alpha \beta} \\
&=& a_{k \gamma} b_{\alpha i} u^{ij}_{\quad kl} b_{\beta j} a_{l \delta} a_{\bar{i} \alpha} b_{\gamma \bar{k}} u^{\bar{k} \bar{l}}_{\quad \bar{i} \bar{j}} b_{\delta \bar{l}} a_{\bar{j} \beta}\\
&=& \delta^k_{\bar{k}} \delta^i_{\bar{i}} \delta^j_{\bar{j}} \delta^l_{\bar{l}} u^{ij}_{\quad kl} u^{\bar{k} \bar{l}}_{\quad \bar{i} \bar{j}}\\
&=& u^{ij}_{\quad kl} u^{kl}_{\quad ij} \\
&=& |F|^2.
\end{eqnarray*}

Hence $$|F|^2 = \sum_{i,j,k,l} \left ( {v^{ij}}_{kl} \right )^2, $$ and
${v^{11}}_{11} \leq |F|.$ Now we have $$v^{11}=\frac{V_{11}}{\det(v_{ij})},$$ where $(V_{ij})$ is the cofactor matrix of $(v_{ij})$. So
$$v^{11}-v^{-1}_{11} = \frac{v_{11}V_{11}-\det(v_{ij})}{v_{11}\det(v_{ij})}.$$ Since $v_{ij}, i \neq j$ vanishes at the
point $p$ we have $$\left ( \frac{\partial}{\partial x_1}
\right)^2(v^{11}-v^{-1}_{11})=2\sum_{k=2}^n\frac{(v_{1k1})^2}{v_{11}\det(v_{ij})}=2\sum_{k=2}^n(v_{1k1})^2 \geq 0$$ at
$p$. So
$$\left (\frac{\partial}{\partial x_1} \right )^2 v^{-1}_{11} \leq v^{11}_{11} \leq |F|.$$ Hence
$$\left (\frac{\partial}{\partial x_1} \right )^2 u^{-1}_{11} \leq u^{11}_{11} \leq |F|.$$
\end{proof}

\begin{lemma}
\label{Upper bound of Hessian u} Let $p$ be a point of $P$ and $\nu=(\nu^i)$ a unit vector. Suppose that the segment
${p+t\nu:-3R \leq t \leq 3R}$ lies in $|P|$, that $|F| \leq 1$ in $P$ and that $u$ satisfies the $M$-condition. Then
$$u_{ij}\nu^i\nu^j \leq \max \left ( \frac{2M}{\pi R}, 2\left ( \frac{M}{\pi} \right )^2 \right ).$$
\end{lemma}

\begin{proof}
We can suppose that $\nu$ is the unit vector in the $x_1$ direction and that $p$ is the origin. Let $H(t)=u_{11}(t,0).$
We apply the definition of the $M$-condition to obtain $$\int^R_{-R} H(t) dt \leq M.$$ By the previous Lemma,
$$\frac{d^2}{dt^2}H(t)^{-1} \leq 1.$$ Suppose $H(0)^{-1}=\epsilon.$ Then $$H(t)^{-1} \leq \epsilon + C t +
\frac{t^2}{2}, $$ where $C=H'(0)$. Thus $$H(t)+H(-t) \geq \frac{1}{\epsilon+Ct+t^2/2} + \frac{1}{\epsilon-Ct+t^2/2}
\geq \frac{2}{\epsilon + t^2/2}.$$ This gives $$\int_{-R}^R H(t) \geq \int^R_{-R} \frac{dt}{\epsilon + t^2/2} =
2\epsilon^{-1/2} \int_0^{R\epsilon^{-1/2}} \frac{dt}{1+t^2/2}.$$ So we have $$M \geq
\frac{2\sqrt{2}}{\sqrt{\epsilon}}\tan^{-1}\left( \frac{R}{\sqrt{2\epsilon}} \right).$$ Now use the fact that
$$\frac{4}{\pi}\tan^{-1}(z) \geq \min(1,z)$$ and a little manipulation to obtain the stated bounds on
$\epsilon^{-1}=u_{11}(0,0).$
\end{proof}

The results in the rest of this subsection depend upon a special feature of the Riemannian metric $g$. We need a simple comparison result for Jacobi fields.

\begin{lemma}
\label{Eigenvalue decreasing} Suppose $\gamma(t) \in P$ is a geodesic parameterized by its arc length $t \in (0,a)$. Then for any vector $\nu_1 \in \mathbb{R}^n$, $$\frac{\sqrt{(u^{ij})(\nu_1, \nu_1)}}{\sinh t}$$ is a decreasing function of t.
\end{lemma}

\begin{proof}
Notice that the Riemannian metric on the whole manifold is invariant under the $\mathbb{T}^n$ action, hence we can let $\nu_1$ be the Jacobi vector field along $\gamma$. Let us fix a point $t_0 \in (0,a)$; without loss of generality, we can assume $\nu_1$'s Riemannian magnitude is 1 at $\gamma(t_0)$, i.e. $|\nu_1|^2=(u^{ij})(\nu_1,\nu_1)=1$. Let us consider the derivative of $|\nu_1| / \sinh t$ at $t=t_0$. We only need to show that $\langle \nu_1', \nu_1 \rangle \leq \cosh t_0/ \sinh t_0$, where $\nu_1'=\nabla_{X(t_0)} \nu_1, X=\frac{\partial \gamma}{\partial t}$.

Let us pick $\nu_2, \ldots, \nu_n \in \mathbb{R}^n$, such that $|\nu_2|=\cdots=|\nu_n|=1$ and $0=\langle \nu_i, \nu_j \rangle =(u^{ij})(\nu_1,\nu_2), i \neq j$ at $\gamma(t_0)$. Assume $e_1, \dots, e_n$ are orthogonal frame along $\gamma$ such that $e_1(t_0)=\nu_1(t_0), e_1'=0, \ldots, e_n(t_0)=\nu_n(t_0), e_n'=0$. Expressing $\nu_1, \ldots,  \nu_n$ in terms of $e_1, \ldots, e_n$, we get $$\nu_i=\nu_i(t)=G^j_i(t) e_j(t). $$

The Jacobi equation tells us that
$$
\nu_i'' + Rm(\nu_i, X)X=0
$$
Let $H$ be a $n \times n$ symmetric matrix such that $H(e_i, e_j)=Rm(e_i,X,X,e_j)$, then
$$
G''+ GH=0.
$$

Let $S = G^{-1}G'$, so that $S$ satisfies the Riccati equation

$$
S'+S^2=-H.
$$

Notice that

$$
\nu_i \langle \nu_j, X \rangle = 0 = \nu_j \langle \nu_i, X \rangle;
$$
we get

\begin{eqnarray*}
\langle \nu_j, \nabla_{\nu_i}X \rangle &=& \langle \nu_i, \nabla_{\nu_j} X \rangle \\
\langle \nu_j, \nu_i' \rangle &=& \langle \nu_i, \nu_j' \rangle \\
G_j^k {G'}_i^k &=& G_i^k {G'}_j^k.
\end{eqnarray*}
Hence $S(t)$ is a symmetric matrix for all $t$. At $t=t_0$ we have $\langle \nu_1', \nu_1 \rangle = S_{11}$, the $(1,1)$ entry of the matrix $S$, so it suffices to prove that all the eigenvalues of $S(t_0)$ are bounded above by $\cosh t_0 / \sinh t_0$. Now each eigenvalue $\lambda(t)$ of $S(t)$ satisfies a scalar Riccati differential inequality

$$
\lambda'+\lambda^2 \leq 1.
$$

By standard arguments, we may ignore the complications that might occur from multiple eigenvalues. Suppose by contradiction that $\lambda(t_0) > \cosh t_0/ \sinh t_0$. Then we can find $\tau \in (0, t_0)$ such that $\lambda(t_0) = \cosh(t_0 - \tau)/ \sinh(t_0 - \tau)$. Now the function $\mu(t) = \cosh(t - \tau)/ \sinh(t - \tau)$ satisfies the equation $\mu' + \mu^2 = 1$. So $\lambda'- \lambda^2 \leq \mu' - \mu^2$ in the interval $(\tau, t_0]$ and $\lambda(t_0) = \mu(t_0)$. It follows that $\lambda(t) \geq \mu(t)$ for $t \in (\tau, t_0)$ and since $\mu(t) \rightarrow \infty$ as $t$ tends to $\tau$ from above we obtain a contradiction.

\end{proof}

\begin{lemma}
Let $E$ be a face of the polytope $P$ and suppose that the defining function $\lambda_E$(determined by $\sigma$) is
$x_1$. Then if $u$ satisfies the Guillemin boundary conditions and $|F| \leq 1$ throughout $P$ we have $$u^{11}(p) \leq
\sinh^2 {\rm Dist}_g(p,E)$$ for any $p$ in $P$.
\end{lemma}

\begin{proof}
To see this we consider a geodesic parameterized by $t \geq 0$, starting at time $0$ on the boundary component $E$.
Near the boundary we can describe the geometry in terms of a $2n$-manifold with a group action in the familiar way. The
vector field $\frac{\partial}{\partial \theta_1}$ is smooth in the $2n$-manifold and vanishes at $t=0$. The condition
that $x_1$ is the normalized defining function just asserts that this vector field is the generator of a circle action
of period $2\pi$. It follows that $$\lim_{t \to 0} t^{-1} |\frac{\partial}{\partial \theta_1}| \leq 1,$$ (with equality
when the geodesic is orthogonal to the edge $E$). Then, by the above lemma, $\sqrt{u^{11}} = |\frac{\partial}{\partial
\theta_1}| \leq \sinh t$ and the result follows.
\end{proof}

\begin{cor}
\label{Corollary of distance control} Let $E$ be a face of $P$ with defining function $\lambda_E$. Then if $|F| \leq
1$ we have $$\lambda_E(p) \leq \cosh ({\rm Dist}_g(p,E))-1.$$
\end{cor}

\begin{proof}
Notice that this is an affine-invariant statement. There is no loss in supposing that, as above, $\lambda_E=x_1$. Then
for a geodesic starting from a point of $E$, parameterized by arc length, we have $$|\frac{dx_1}{dt}| \leq |dx_1|_g =
\sqrt{u^{11}} \leq \sinh t$$ hence $x_1 \leq \cosh t-1.$
\end{proof}

\begin{lemma}
\label{Control of Hessian of u}
Suppose that $|F| \leq 1$ and that $p$ is a point in $P$ with ${\rm Dist}_g(p,\partial P) \geq \alpha >0.$ Then if $q$
is a point with ${\rm Dist}_g(p,q)=d$ we have $$(u^{ij}(q)) \leq \frac{\sinh^2(\alpha+d)}{\sinh^2 \alpha}
(u^{ij}(p)).$$ If $d<\alpha$ we have $$(u^{ij}(q)) \geq \frac{\sinh^2(\alpha-d)}{\sinh^2 \alpha}(u^{ij}(p)).$$
\end{lemma}

\begin{proof}
To prove the Lemma, observe that it suffices by affine invariance to prove the corresponding inequalities for the
matrix entry $u^{11}=|\epsilon_1|^2.$ For the first inequality we consider a minimal geodesic $\gamma$ from
$p=\gamma(0)$ to $q=\gamma(d)$ and extend it "backwards" to $t > -\alpha$. Then replacing $t$ by $t+\alpha$ we are in
the situation considered in Lemma \ref{Eigenvalue decreasing} and we obtain $$\frac{|\epsilon_1(p)|}{\sinh \alpha} \geq
\frac{|\epsilon_1(q)|}{\sinh(\alpha+d)}.$$ For the second inequality we extend the geodesic "forwards" to the interval
$[0,\alpha]$ and argue similarly.
\end{proof}

Suppose that $p=(p^1,\ldots,p^n)$ is a point of $P$ and $r > 0$. Put $$E_{p,r}=\{(x_1,\ldots,x_n) \in \mathbb{R}^n :
u_{ij}(p)(x_i-p^i)(x_j-p^j) \leq r^2 \}.$$ So $E(p,r)$ is the interior of the ellipse defined by the parameter $r$ and
the quadratic form $u_{ij}(p).$ The Euclidean area of $E(p,r)$ is $\det (u_{ij}(p))^{-1/2}\omega_n r^n$, where $\omega_n$
is the volume of a standard Euclidean $n$-ball.

\begin{lemma}
\label{Elliptic balls} Suppose that $|F| \leq 1$ and that $p$ is a point in $P$ with ${\rm Dist}_g(p, \partial P) \geq
\alpha >0.$ Then for any $\beta < \alpha$ the $\beta$-ball in $P$, with respect to the metric $g$ satisfies $$E(p,
c\beta) \subset B_g(p,\beta) \subset E(p,C\beta),$$ where $c=\sinh (\alpha-\beta)/\sinh \alpha$ and
$C=\sinh(\alpha+\beta)/\sinh \alpha$. In particular, the Euclidean area of the $\beta$ ball for the metric $g$ is
bounded below by $$\mbox{Area}_{Euc}B_g(p,\beta) \geq c^n \beta^n \omega_n \det(u_{ij})(p)^{-1/2}.$$
\end{lemma}

\begin{proof}
There is no loss in supposing that the matrix $u^{ij}(p)$ is the identity matrix, so we have to show that the ball
$B_g(p,\beta)$ defined by the metric $g$ contains a Euclidean disc of radius $c\beta$, and is contained in a Euclidean
disc of radius $C\beta$. We know by the above lemma that on the ball $B_g(p,\beta)$ we have $$c^2 \leq (u^{ij}) \leq
C^2.$$ Thus $C^{-2} \leq (u_{ij}) \leq c^{-2}$, and the Euclidean length of a path in $B_g(p,\rho)$ is at least $c^{-1}$
times the length calculated in the metric $g$, and at most $C^{-1}$ times that length. The second statement immediately
tells us that $B_g(p, \beta)$ lies in $E(p,C\beta)$. In the other direction, suppose $q$ is a point in the Euclidean disc
of radius $c\beta$ centered on $p$. We claim that $q$ lies in the (closed) $g$ ball $B_g(p,\beta)$. For if not there is
a point $q'$ in the open line segment $pq$ such that the distance from $q'$ to $p$ is $\beta$ and the line segment
$pq'$ lies in $B_g(p,\beta)$. But the Euclidean length of this line segment is strictly less than $c\beta$ so the
length in the metric $g$ is less than $\beta$, a contradiction.
\end{proof}

\section{Singularity Analysis}
By Chen-He's result \cite{ChenHe}, the Calabi flow exists for a short time. Suppose that the Calabi flow does not exist for all time and the singular time is $T$, i.e, the Riemannian curvature blows up at time $T$. We will use blowing up arguments to rule out different kinds of singularities under the Calabi flow. 

Since $\frac{\partial u_i}{\partial t}=-A_i$ and $|\nabla A|$ is bounded for all $t < T$, we conclude that for any $ t < T$, $u(t)$ satisfies a $M$-condition by Lemma \ref{Init}. Next we want to show that the scalar curvature $A$ is bounded for any finite time $T$. It is well known that the Calabi energy, $L^2$-norm of $A - \underline{A}$,  is decreasing under the Calabi flow. We want to see the corresponding formula in toric case. In fact, by the Abreu's formula, we have $$A=-\sum_{ij}\frac{\partial^2 u^{ij}}{\partial x_i \partial x_j},$$
hence $$\frac{d A^2}{d t}=- 2 A E^{ij}_{\ ij},$$
where $$E^{ij} = \frac{\partial u^{ij}}{\partial t}.$$ So
\begin{eqnarray*}
\frac{d}{dt}\int_P A^2 d \mu & = & - 2\int_P (A - \underline{A}) E^{ij}_{ij} d\mu \\
& = & -2\int_{\partial P} (A - \underline{A}) E^{ij}_i \nu_j d \sigma + 2 \int_P A_j E^{ij}_i d \mu \\
& = & 2 \int_P A_j E^{ij}_i d \mu\\
& = & 2 \int_{\partial P} A_j E^{ij} \nu_i d\sigma - 2 \int_P A_{ij}E^{ij} d\mu \\
& = & -2 \int_P A_{ij} E^{ij} d \mu \\
& = & -2 \int_P A_{ij} u^{ia} A_{ab} u^{bj} d \mu \leq 0.\\
\end{eqnarray*}

In the above calculations, the integral domains in fact are $P_{\delta}$ and $\partial P_{\delta}$ where $P_{\delta}$ is an interior domain whose boundary has a distance $\delta$ from the boundary of $P$ and we let $\delta \rightarrow 0$. The reason why all the boundary integrals go to 0 relies on the Guillemin boundary condition: without loss of generality, we can set $\nu =\langle 1,0,\ldots,0 \rangle$, so $u=x_1 \log x_1 + f$ where $f$ is a smooth function up to the boundary. Then $u^{1j}$ are all products of $x_1$ with smooth functions. So $E^{1i}=0, E^{1j}_j=0$ for all $i$ and $j>1$. Since the boundary measure $d\sigma$ is fixed, $u^{11}_1$ is fixed, hence $E^{11}_1=0$.\\

If the scalar curvature is not bounded, then there is $t < T$ and $x_0 \in P$, such that $|A(x_0,t)|>C_1$. By the assumption that $|\nabla A|$ is uniformly bounded at $[0,t]$, there is a neighborhood $x_0 \in Q \subset P$ with $Vol(Q) > C_2$ and $|A(x,t)|> C_1/2$ for all $x \in Q$. Then $\int_P A^2 > C_1^2 C_2/4$ at time t. It is easy to see that we can get $C_1$ as large as we want with $C_2$ fixed which contradicts the fact that the Calabi energy is decreasing along the Calabi flow. \\

Since the Calabi flow cannot extend through time $T > 0$, the L$^{\infty}$-norm of Riemannian curvature of $t$-slice would blow up as $t \rightarrow T$. Now pick a sequence of points $(p_i,t_i) \rightarrow (p,T)$ where $|Rm(p_i,t_i)|$ realizes the L$^{\infty}$ Riemannian curvature norm at $t_i$. We want to show the rescaling process in the corresponding polytope. Here we follow Donaldson's work \cite{D3}. Suppose $u$ is a convex function on a polytope $P$ with $u^{ij}_{ij}=-A$. Let $\lambda$ be a positive real number. Define a function $\tilde{u}$ on the polytope $\tilde{P}=\lambda P$ by $$\tilde{u}(x_1,\ldots,x_n)=\lambda u(\lambda^{-1}x_1,\ldots,\lambda^{-1}x_n).$$

\begin{prop}[Donaldson] $\tilde{u}$ satisfies the following properties:

\begin{itemize}
\item The curvature $\tilde{F}$ of $\tilde{u}$ satisfies $$|\tilde{F}|(\lambda p)=\lambda^{-1} |F|(p).$$
\item The scalar curvature $\tilde{A}=-{\tilde{u}^{ij}}_{ij}$ is $$\tilde{A}(\lambda p)=\lambda^{-1} A(p).$$
\item If $u$ satisfies an $M$-condition then so does $\tilde{u}$ (with the same value of $M$).
\end{itemize}
\end{prop}
\begin{proof}
Using the chain rule for partial derivatives, we obtain ${\tilde{u}^{ij}}_{\ \ kl}=\lambda^{-1} {u^{ij}}_{kl}$ and hence the desired conclusion.
\end{proof}

Dilating by a factor $\lambda_i = |Rm(p_i,t_i)|$, we get a new sequence of data $(\tilde{P}^{(i)}, \tilde{u}^{(i)})$. It is clear that, perhaps after taking a subsequence, one of the two cases must occur.
\begin{itemize}
\item The limit of the $\tilde{P}^{(i)}$ is the whole of $\mathbb{R}^n$;
\item The limit of the $\tilde{P}^{(i)}$ is a $(\mathbb{R}^+)^m \times \mathbb{R}^{n-m}$ with $m < n$.
\end{itemize}

To rule out those singularities, the idea is to study the equation
$$
-\sum u^{ij}_{ij} = A.
$$
Donaldson \cite{D2} rewrites the above equation in the following form
$$
-U^{ij} \left( \frac{1}{\det(u_{kl})} \right)_{ij}=A,
$$
where $(U^{ij})$ is the cofactor matrix of $(u_{ij})$.

It is easy to check that if we can show that for any compact set away from the boundary, the operator $(U^{ij})$ is uniformly elliptic, then the limit equation will be
$$
-U^{ij} \left( \frac{1}{\det(u_{kl})} \right)_{ij}=0
$$
in the weak sense. Thus we can use the maximal principle to obtain a contradiction. To control the upper and lower bound of $(u_{ij})$, we  utilize Donaldson's idea \cite{D3}. 

{\it Case 1: The limiting domain is $\mathbb{R}^n$.}\\

Let us normalize $u^{(i)}$ first. By translation we can assume that each $p_i$ is the origin and $u^{(i)}$ is normalized at the origin, i.e, $u^{(i)}(0)=0, \nabla(u^{(i)})=0$. We want to show that on any compact subset $K \subset \mathbb{R}^n$ we have upper and lower bounds $$C_K^{-1} \leq \tilde{u}^{(i)}_{ij} \leq C_K.$$ The upper bound follows immediately from Lemma \ref{Upper bound of Hessian u} (since on compact sets the Euclidean distance of the boundary of $\tilde{P}^{(i)}$ tends to infinity with $i$). Let $J=J^{(i)}$ be the function $\det (\tilde{u}_{ij}).$ The crucial thing is to get a lower bound on $J(0)$. Corollary \ref{Corollary of distance control} implies that the distance in the metric $\tilde{g}^{(i)}$ corresponding to $\tilde{u}^{(i)}$ from the origin to the boundary of $\tilde{P}^{(i)}$ tends to infinity. By construction, $|\widetilde{Rm}^{(i)}|$ is equal to 1 at the origin. Since the derivative of Riemannian curvature is uniformly bounded, we can find a fixed small number $\delta$ such that $|\widetilde{Rm}^{(i)}| \geq 1/2$ on the $\tilde{g}^{(i)}$ ball of radius $\delta$ about the origin. On the other hand Lemma \ref{Elliptic balls} implies that this ball contains a Euclidean ellipse of area at least $c J(0)^{-1/2} \delta^n$, for some fixed $c$. Thus $$\int_{\tilde{P}^{(i)}} |\widetilde{Rm}^{(i)}|^n d \mu_{Euc} \geq c \delta^n J(0)^{-1/2}.$$
Since the $L^n$-norm of the Riemannian curvature is bounded and is scaling invariant, the integral on the left is bounded. So we obtain a lower bound on $J(0)$, as required. Combined with the upper bound on $\tilde{u}_{ij}$ this lower bound on $J(0)$ yields an upper bound on $\tilde{u}^{ij}$ at the origin. Now Lemma \ref{Control of Hessian of u} gives an upper bound on $\tilde{u}^{ij}$ at points of bounded $\tilde{g}$ distance from the origin. The upper bound on $\tilde{u}_{ij}$ implies that on compact subsets of $\mathbb{R}^n$ the $\tilde{g}$ distance to the origin is bounded. So we conclude that $\tilde{u}^{ij}$ is bounded above on compact subsets of $\mathbb{R}^n$. Also $\tilde{u}^{ij}$ is bounded below on compact subsets of $\mathbb{R}^n$. Once we have these upper and lower bounds on $\tilde{u}^{ij}$ the convergence of a subsequence is straightforward, and the limit function $\tilde{U}$ satisfies ${\tilde{U}^{ij}}_{\quad ij}=0$. However, the following theorem tells us that it cannot happen. \\

Note that Donaldson proves a stronger result in the case of n = 2 and
Trudinger-Wang study a similar fourth order PDE in their work \cite{TW2}. However,
our approach is different than theirs.

\begin{prop}
There is no convex function $u$ satisfying the following conditions simultaneously: 1. ${u^{ij}}_{ij}=0$. 2. $|F| \leq 1.$ 3. $ |\nabla u| < M$.
\end{prop}

\begin{proof}
Let $G=\frac{1}{\det(u_{ij})}$, then $G$ satisfies $U^{ij}G_{ij}=0$ where $U^{ij}$ is the cofactor matrix of $u_{ij}$. Applying Lemma \ref{Upper bound of Hessian u}, there is a constant $C$ such that ${\rm Hess(u)}(v,v) < C$ for any unit vector $v$. And for arbitrary $\epsilon > 0$, we know that for $|x|$ large enough, $$\int_{-1}^1 {\rm Hess}(u)|_{x+tv}(\bar{v}, \bar{v}) dt < 2\epsilon,$$ where $\bar{v}$ is the unit vector at $x$ pointing away from the origin. Using the same tricks as we did in Lemma \ref{Upper bound of Hessian u}, we get $${\rm Hess}(u)|_x(\bar{v},\bar{v}) \leq \max\left(\frac{4\epsilon}{\pi}, \left( \frac{4\epsilon}{\pi} \right)^2 \right).$$

Now it is clear that for any $\epsilon > 0$, there is a $R > 0$, such that for any $x$ outside $B(0,R)$, $|\det(u_{ij})(x)|< \epsilon$. That tells us that $G$ reaches its minimum in the interior of $\mathbb{R}^n$. Since $G$ satisfies $U^{ij}G_{ij}=0$, we know that $G$ must be a constant; hence it must be zero, a contradiction.
\end{proof}

{\it Case 2: The limiting domain is $(\mathbb{R}^+)^m \times \mathbb{R}^{n-m}, m>1$.}

In those cases, we still follow Donaldson's work to get the uniform ellipticity of the operator $(U^{ij})$ and use a different method to rule out the singularities. We pick a point $\tilde{p}_i$ such that it satisfies the following two conditions.
\begin{itemize}
\item The Euclidean distance between $\tilde{p}_i$ and $p_i$ is bounded and hence its Riemannian distance is also bounded by Lemma \ref{Riemannian distance controlled by Euclidean distance}. Since the derivative of the Riemannian curvature is controlled, we can assume $|Rm(\tilde{p}_i)| \leq 1/2$.
\item $\tilde{p}_i$ is not close to the boundary in the sense of the Euclidean distance and hence it is not close to the boundary in the sense of the Riemannian distance by Corollary \ref{Corollary of distance control}.
\end{itemize}

The remaining process is the same as in the previous case, we normalize $\tilde{u}^{(i)}$ at $\tilde{p}_i$ and there is a subsequence converging to a smooth function $U$ satisfying ${U^{ij}}_{ij}=0, |F^{(\infty)}|^2={U^{ij}}_{kl} {U^{kl}}_{ij} \leq 1$ and $M$-condition. Notice that the function $f(x)=x \log x$ satisfies $f'(x/2)-f'(x)=-\log 2, x>0$. Then for any positive number $c$, there is a $D > 0$ such that $f'(x/D)-f'(x) < - c$. Based on this observation, let us check what happens in our limiting function $U$. Without loss of generality, we can assume one edge of $P$ is $\{ x_1=0 \}$ and $\{ x_1 > 0 \} \cap P$ is not empty. We want to consider the derivative of $u^{(i)}$ in the $x_1$ direction near the boundary $x_1=0$. Since $|\nabla A|$ is bounded, what really matters is $x_1 \log x_1$. Hence for any $c>0$, there is a uniform constant $D > 0$, such that
$$\frac{\partial u^{(i)}}{\partial x_1} (x_1/D, x_2, \ldots, x_n) - \frac{\partial u^{(i)}}{\partial x_1} (x_1, \ldots, x_n) < -c$$ 
for all functions $u^{(i)}$. Since we normalize $\tilde{u}^{(i)}$ at $\tilde{p}_i$ whose Euclidean distance to the boundary is bounded from below, then for any positive constant $c > 0$, there is a uniform constant $d > 0$ such that for every point $x$ and for any $j \leq m$, if its $j$-th coordinate is less than $d$, then $$\frac{\partial \tilde{u}^{(i)}}{\partial x_j}(x) < -c.$$ The same conclusion holds for our limiting function $U$. In the next proposition, we will show that it is impossible.

\begin{prop}
There is no convex function $u$ defined on $(\mathbb{R}^+)^m \times {\mathbb{R}}^{n-m}$ satisfying the following conditions simultaneously:\
\begin{itemize}
\item ${u^{ij}}_{ij}=0$
\item $|F|^2={u^{ab}}_{cd} {u^{cd}}_{ab} \leq 1$
\item $u$ satisfies the $M$-condition.
\item For every positive number $C > 0$, there is a $d > 0$ such that for every point $x$ and for any $i \leq m$, if its $i$-th coordinate is less than $d$, then $$\frac{\partial u}{\partial x_i} < -C$$
\end{itemize}
\end{prop}

\begin{proof}
We will check the $\Omega_1 = \mathbb{R}^{+} \times \mathbb{R} \times \cdots \times \mathbb{R}$ case first. Let us consider the Legendre dual of $u$, i.e, $\phi$. We have $$0= \phi^{ij}(\log \det (\phi_{ab}))_{ij}$$ and in the coordinate system $$y_i=\frac{\partial u}{\partial x_i}.$$ The corresponding region $\Omega_2$ of $\Omega_1$ in the $y_i$ space will be a tube, where $y_2, \ldots, y_n$ are bounded and $y_1$ has an upper bound but no lower bound. Let us pick a constant $c$ sufficiently small such that $\Omega_4=\{y_1 > c \} \cap \Omega_2 \neq \o$. By the fourth assumption, there is a constant $d > 0$ such that the preimage of $\Omega_4$ will be contained in $\Omega_3=\Omega_1 \cap \{x_1 \geq d\}$.\\

Without loss of generality, we can pick a point $\bar{y} \in \Omega_4$ such that its first coordinate $\bar{y}_1 > c$. Let $G(y)=\log\det(\phi_{ab}) - \lambda (y_1-\bar{y}_1)$, where $\lambda$ is a constant to be determined later. The idea is that if $G(y)$ reaches its minimum in the interior of $\Omega_4$, the equation $$0=\phi^{ij}G_{ij}$$ tells us that in any compact set of $\Omega_4$, $G$ is a constant. Hence $\log\det(\phi_{ab})=\lambda(y_1-\bar{y}_1)+C$ on any compact set of $\Omega_4$. However, $\log\det(\phi_{ab})$ approaches infinity at some boundaries of $\Omega_4$ and we would get a contradiction.\\

The remaining task is to pick an appropriate $\lambda$ such that $G(y)$ reaches it minimum in the interior of $\Omega_4$. Suppose $y \in \partial \Omega_4$ and the first coordinate of $y$ is strictly greater than $c$. We claim that $G(y)$ is infinity in this case. Let us pick a sequence of points $y_n \in \Omega_4$ approaching $y$. It is easy to see that $y$ is also in the boundary of $\partial \Omega_2$, since the mapping from $\Omega_4$ to $\Omega_3$ is local diffeomorphism, the corresponding $x_n \in \Omega_3$ of $y_n$ must approach infinity. Because the first coordinate of $x_n$ is greater than $d$, for any unit vector $v$, ${\rm Hess}(u)(v,v)$ is bounded from above by Lemma \ref{Upper bound of Hessian u}. And since $x_n$ approaches infinity, ${\rm Hess}(u)(v,v)$ approaches 0 where $v$ is the unit vector at $x_n$ pointing away from $(d,0,\ldots,0)$. So $\det(u_{ij})(x_n)$ approaches 0. Hence we can conclude that $\det(\phi_{ab})(y_n)$ approaches infinity. \\

What left is that the first coordinate of $y$ is equal to $c$. Since $\det(u_{ij})$ is bounded from above in $\Omega_3$, $\det(\phi_{ab})$ is bounded from below in $\Omega_4$. Hence we can find a $\lambda$ big enough such that for all such $y$, $G(y) > G(\bar{y})$.\\

For the other cases, they are almost the same. We let $\Omega_4 = \Omega_2 \cap \{ y_1 > c_1 \} \cap \cdots \cap \{ y_m > c_m \}$ and let $G(y)=\log \det (\phi_{ab}) - \lambda_1(y_1-\bar{y}_1) - \cdots - \lambda_m(y_m-\bar{y}_m)$. Then we can show that $G(y)$ reaches its minimum in the interior of $\Omega_4$ by picking $\lambda_1, \ldots, \lambda_m$ appropriately.
\end{proof}

By the above arguments, we complete the proof of Theorem (\ref{Calabi flow}).

Hongnian Huang, \ hnhuang@gmail.com

Centre interuniversitaire de recherches en geometri et topologie

Universite du Quebec a Montréal

Case postale 8888, Succursale centre-ville

Montreal (Quebec)


\begin{thebibliography}{1}
\bibitem{A1} M. Abreu, {\em K\"ahler geometry of toric varieties and extremal metrics}, International
J. Math. 9 (1998) 641-651.

\bibitem{A2} M. Abreu, {\em K\"ahler metrics on toric orbifolds}, J. Differential Geom. 58 (2001), no. 1, 151-187

\bibitem{ACG} V. Apostolov, D. Calderbank, P. Gauduchon, {\em Hamiltonian 2-forms in Kaehler Geometry I: General Theory}, J. Differential Geom. 73 (2006), 359-412.

\bibitem{Ca1} E. Calabi, {\em Extremal K\"ahler metric, in Seminar of Differential Geometry},
ed. S. T. Yau, Annals of Mathematics Studies 102, Princeton University
Press (1982), 259-290.

\bibitem{Ca2} E. Calabi, {\em Extremal K\"ahler metric, II, in Differential Geometry and Complex
Analysis}, eds. I. Chavel and H. M. Farkas, Spring Verlag (1985), 95-114.

\bibitem{Chen1} X. X. Chen, {\em Calabi flow in Riemann surfaces revisited}, IMRN, 6 (2001), 275-297.

\bibitem{ChenHe} X. X. Chen, W. Y. He, {\em On the Calabi flow}, Amer. J. Math. 130 (2008), no. 2, 539-570.

\bibitem{ChenHe2} X. X. Chen, W. Y. He, {\em The Calabi flow on K\"ahler surface with bounded Sobolev constant--(I)}, arXiv:0710.5159

\bibitem{ChenHe3} X. X. Chen, W. Y. He {\em The Calabi flow on toric Fano surface}, arXiv:0807.3984.

\bibitem{Chr} P. T. Chrusci\'el, {\em Semi-global existence and convergence of solutions of
the Robison-Trautman(2-dimensional Calabi) equation}, Comm. Math. Phys. 137 (1991), 289-313.

\bibitem{D1} S.K. Donaldson, {\em Scalar curvature and stability of toric varieties}, Jour. Differential
Geometry 62 (2002), 289-349.

\bibitem{D2} S.K. Donaldson, {\em Interior estimates for solutions of Abreu°Øs equation}, Collectanea
Math. 56 (2005), 103-142.

\bibitem{D3} S.K. Donaldson, {\em Extremal metrics on toric surfaces: a continuity method}, J. Differential Geom. 79 (2008), no. 3, 389-432.

\bibitem{D4} S.K. Donaldson, {\em Constant scalar curvature metrics on toric surfaces},	Geom. Funct. Anal. 19 (2009), no. 1, 83-136. 

\bibitem{JF} J. Fine, {\em Calabi flow and projective embeddings}, J. Differential Geom. 84 (2010), no. 3, 489-523.

\bibitem{G1} V. Guillemin, {\em Kaehler structures on toric varieties}, J. Differential Geom. 40
(1994), 285-309.

\bibitem{G2} V. Guillemin, {\em Moment maps and combinatorial invariants of Hamiltonian $T^n$-
spaces}, Birkhauser, 1994.

\bibitem{He} W. Y. He, {\em Local solution and extension to the Calabi flow}, arXiv:0904.0978.

\bibitem{Pe} G. Perelman, {\em The entropy formula for the Ricci flow and its geometric
applications}, math.DG/0211159, 2002.

\bibitem{Shi}W. X. Shi, {\em Ricci deformation of the metric on complete noncompact
Riemannian manifolds} , J. Differential Geom, 30(2) (1989), 303-394.

\bibitem{GA} G. Sz\'{e}kelyhidi, {\em The Calabi functional on a ruled surface},  arXiv:math/0703562.

\bibitem{TW2} N. S. Trudinger and X-J. Wang, {\em The Bernstein problem for affine maximal
hypersurfaces}, Invent. Math., 140 (2000), 399-422.
\end{thebibliography}
\end{document}